\documentclass[12pt]{article}
\usepackage{latexsym}
\usepackage{amsmath,amsfonts,amssymb,mathrsfs,graphicx}
\usepackage{amsthm}
\usepackage{xcolor}
\usepackage{verbatim}

\title{Intermediate dimensions - a survey}
\author{Kenneth J. Falconer}

\date{}

\def\rn{\mathbb{R}^n}

\renewcommand{\epsilon}{\varepsilon}

\newcommand\ubd{\overline{\mbox{\rm dim}}_{\rm B}\,} 
\newcommand\lbd{\underline{\mbox{\rm dim}}_{\rm B}\,} 
\newcommand\da{\underline{\mbox{\rm dim}}_{\,\theta}} 
\newcommand\db{\underline{\mbox{\rm dim}}_{\,\phi}} 
\newcommand\udb{\overline{\mbox{\rm dim}}_{\,\phi}} 
\newcommand\uda{\overline{\mbox{\rm dim}}_{\,\theta}} 
\newcommand\dth{{\mbox{\rm dim}}_{\,\theta}} 
\newcommand{\lid}{\underline{\dim}_{\,\theta}}
\newcommand{\uid}{\overline{\dim}_{\,\theta}}
\newcommand\bdd{\mbox{\rm dim}_{\rm B}}
\newcommand\hdd{\mbox{\rm dim}_{\rm H}\,} 
\newcommand{\dimh}{\dim_{\rm H}}

\newcommand{\be}{\begin{equation}} 
\newcommand{\ee}{\end{equation}} 

\DeclareMathOperator*\lowlim{\liminf}
\DeclareMathOperator*\uplim{\limsup}

 \newtheorem{theo}{Theorem}[section]
 \newtheorem{cor}[theo]{Corollary}
 
 \newtheorem{prop}[theo]{Proposition}

 \newtheorem{defi}[theo]{Definition}

\newtheorem{question}[theo]{Question}

\begin{document}
\maketitle

\vspace{-8mm}
\begin{center}
Mathematical Institute,	
University of St Andrews,\\
St Andrews,
Fife KY16 9SS,
UK
\medskip

E-mail: \texttt{kjf@st-andrews.ac.uk}
		 	
\end{center}

\begin{abstract}

This article surveys the $\theta$-intermediate dimensions that were introduced recently which provide a parameterised continuum of dimensions that run from Hausdorff dimension when $\theta=0$ to box-counting dimensions when $\theta=1$. We bring together diverse properties of intermediate dimensions which we illustrate by examples.
\medskip

\noindent\emph{Mathematics Subject Classification 2010}: primary: 28A80; secondary: 37C45.

\noindent\emph{Key words and phrases}: Hausdorff dimension, box-counting dimension, intermediate dimensions.
 \end{abstract}

\section{Introduction}
\setcounter{equation}{0}
\setcounter{theo}{0}

Many interesting fractals, for example many self-affine carpets, have differing box-counting and Hausdorff dimensions. A smaller value for Hausdorff dimension can result because covering sets of widely ranging scales are permitted in the definition, whereas box-counting dimensions essentially come from counting covering sets that are all of the same size.  Intermediate dimensions were introduced in \cite{FFK}  in 2019 to provide a continuum of dimensions between Hausdorff and box-counting; this is achieved by restricting the families of allowable covers in the definition of Hausdorff dimension by requiring that $|U| \leq |V|^\theta$ for all sets $U, V$  in an admissible cover, where $\theta \in [0,1]$ is a parameter.  When $\theta=1$ only covers using sets of the same size are allowable and we recover box-counting dimension, and when $\theta=0$ there are no restrictions giving Hausdorff dimension. 

This article brings together what is currently known about intermediate dimensions from a number of sources, especially \cite{Ban,Bur,BFF,FFK,Kol}; in particular Banaji \cite{Ban} has very recently obtained many detailed results.  We first consider basic properties of $\theta$-intermediate dimensions, notably continuity  when $\theta \in (0,1]$, and discuss some tools that are useful when working with intermediate dimensions. We look at some examples to show the sort of behaviour that occurs, before moving onto the more challenging case of Bedford-McMullen carpets. Finally we consider a potential-theoretic characterisation of intermediate dimensions which turns out to be useful for studying the dimensions of projections and other images of sets. Proofs for most of the results can be found elsewhere and are referenced, though some are sketched to provide a feeling for the subject. 

We  work with subsets of $\rn$ throughout, although much of the theory easily extends to more general metric spaces, see \cite{Ban}. To avoid problems of definition, we assume throughout this account that all the sets $F\subset \rn$ whose dimensions are considered are non-empty and bounded.

Whilst Hausdorff dimension $\dimh$ is usually defined via Hausdorff measure, it may also be defined directly, see \cite[Section 3.2]{Fa}.
For  $F\subset \rn$ we write  $|F|$ for the  {\em diameter} of  $F$ and say that  a finite or countable collection of subsets $\{U_i\}$ of $\rn$  is a {\em cover} of $F$ if $F\subset \bigcup_i U_i$. Then the Hausdorff dimension of $F$ is given by:
\begin{align*}\label{hdim}
\dimh F = \inf \big\{& s\geq 0  :  \mbox{ \rm for all $\epsilon >0$ there exists a cover $ \{U_i\} $ of $F$ such that $\sum_i |U_i|^s \leq \epsilon$}  \big\}.
\end{align*}
(Lower) box-counting dimension $\lbd$ may be expressed in a similar manner except that here we require the covering sets all to be of equal diameter.  For  bounded $F\subset \rn$,
\begin{align*}
\lbd F =  \inf \big\{ s\geq 0  &:  \mbox{ \rm for all $\epsilon >0$ there exists a cover $ \{U_i\} $ of $F$} \\
 & \mbox{ \rm  such that  $|U_i| = |U_j|$ for all $i,j$  and  $\sum_i |U_i|^s \leq \epsilon$}  \big\}.
\end{align*}
From this viewpoint, Hausdorff and box-counting dimensions may be regarded as extreme cases of the same definition, one with no restriction on the size of covering sets, and the other requiring them all to have equal diameters;  
 one might regard these two definitions as the extremes of a continuum of dimensions with increasing restrictions on the relative sizes of covering sets.  This motivates the definition of  {\em  intermediate dimensions} where the coverings are restricted by requiring the diameters of the covering sets to lie in a geometric range $\delta^{1/\theta} \leq |U_i| \leq \delta$  where $0\leq \theta \leq 1$ is a parameter.

\begin{defi}\label{adef}
Let $F\subset \rn$. For $0\leq \theta \leq 1$ the {\em  lower $\theta$-intermediate dimension} of $F$ is defined by
\begin{align*}
\da F =  \inf \big\{& s\geq 0  :  \mbox{ \rm for all $\epsilon >0$ and all $\delta_0>0$,  there exists $0<\delta\leq \delta_0$} \\
 & \mbox{ \rm and a cover $ \{U_i\} $ of $F$ such that  $\delta^{1/\theta} \leq  |U_i| \leq \delta$ and 
 $\sum |U_i|^s \leq \epsilon$}  \big\}.
\end{align*}
Analogously the {\em  upper $\theta$-intermediate dimension} of $F$ is defined by
\begin{align*}
\uda F =  \inf \big\{& s\geq 0  :  \mbox{ \rm for all $\epsilon >0$ there exists $\delta_0>0$ such that for all $0<\delta\leq \delta_0$,} \\
 & \mbox{ \rm there is a cover $ \{U_i\} $ of $F$ such that  $\delta^{1/\theta} \leq  |U_i| \leq \delta$ and 
 $\sum |U_i|^s \leq \epsilon$}  \big\}.
\end{align*}
\end{defi}

\noindent Note that, apart from when $\theta =0$, these definitions are unchanged if $\delta^{1/\theta} \leq  |U_i| \leq \delta$ is replaced by  $\delta \leq  |U_i| \leq \delta^{\theta}$.
\smallskip

It is immediate  that 
$$\hdd F=\underline{\mbox{\rm dim}}_{0} F = \overline{\mbox{\rm dim}}_{0} F, \quad \lbd F = \underline{\mbox{\rm dim}}_{1} F  \quad {\mbox{and}}\quad \ubd F = \overline{\mbox{\rm dim}}_{1} F, $$
where $\ubd$ is upper box-counting dimension.
Furthermore, for a bounded $F\subset \rn$ and $\theta \in [0,1]$,
\[
0\leq \hdd F \leq \da F \leq \uda F \leq \ubd F\leq n  \quad {\mbox{and}}\quad 0\leq \da F \leq  \lbd F\leq n.
\]
As with box-counting dimensions we often have $\da F = \uda F$ in which case we just write $\mbox{\rm dim}_{\theta} F = \da F = \uda F$ for the {\em $\theta$-intermediate dimension} of $F$.

We remark that a continuum of dimensions of a different form, known as the \emph{Assouad spectrum}, has also been investigated recently, see \cite{FraSurA,FraBook, spec1}; this provides a parameterised family of dimensions which interpolate between upper box-counting dimension and quasi-Assouad dimension, but we do not pursue this here.

\section{Properties of  intermediate dimensions}\label{sec:props}
\setcounter{equation}{0}
\setcounter{theo}{0}

\subsection{Basic properties}

We start by reviewing some basic properties of intermediate dimensions of a type that are familiar in many definitions of dimension. 

\begin{enumerate}
\item
{\em Monotonicity. } For all $\theta \in [0,1]$ if $E\subset F$ then $\da E\leq \da F$ and $\uda E\leq \uda F$.

\item
{\em Finite stability. } For all $\theta \in [0,1]$ if $E, F\subset \rn$  then $\uda E\cup F = \max\{\uda E,\uda F\}$. Note that, analogously with box-counting dimensions,  $\da$ is not finitely stable, and neither $\da$ or $\uda$ are countably stable (i.e. it is not in general the case that $\uda\cup_{i=1}^\infty F_i =\sup_{1\leq i<\infty} \uda F_i$).

\item
{\em Monotonicity in $\theta$. }  For all bounded $F$, $\da F$ and $\uda F$ are monotonically increasing in $\theta\in [0,1]$.

\item
{\em Closure. }  For all  $\theta \in (0,1]$, $\da F = \da{\overline F}$ and $\uda F = \uda{\overline F}$ where ${\overline F}$ is the closure of $F$. (This follows since for $\theta \in (0,1]$ it is enough to consider finite covers of closed sets in the definitions of intermediate dimensions.)

\item
{\em Lipschitz and H\"{o}lder properties.  }  Let $f: F\to \mathbb{R}^m$ be an $\alpha$-H\"{o}lder map, i.e. $|f(x)-f(y)|\leq c|x-y|^\alpha$ for $\alpha \in (0,1]$ and $c>0$. Then for all  $\theta \in [0,1]$,  
\begin{equation}\label{Holder}
\da f(F) \leq \frac1\alpha\da F\quad \mbox{and} \quad \uda f(F)\leq \frac1\alpha\uda F.
\end{equation} 
(To see this, if $ \{U_i\} $ is a cover of $F$ with $\delta \leq  |U_i| \leq \delta^{\theta}$  consider the cover of $f(F)$ by the sets $ \{f(U_i)\} $ if $c\delta^\alpha \leq |f(U_i)|$ and by  sets $V_i \supset  f(U_i)$ with $|V_i| =c\delta^\alpha$ otherwise.)

In particular, if $f: F\to f(F) \subset \mathbb{R}^m$ is bi-Lipschitz then  $\da f(F) = \da F$ and $\uda f(F)= \uda F$, i.e. $\da$ and $\uda$ are bi-Lipschitz invariants. For further Lipschitz and H\"{o}lder estimates see Banaji \cite[Section 4]{Ban}.

\end{enumerate}

\subsection{Continuity}

A natural question is whether, for a fixed bounded set $F$, $\da F$ and $\uda F$  vary continuously for $\theta\in [0,1]$. It turns out that this is the case except possibly at $\theta = 0$ where the intermediate dimensions may or may not be continuous, see the examples in Section \ref{egs}.  Continuity on $(0,1]$ follows immediately from the following inequalities which relate $\da F$, respectively $\uda F$, for different values of $\theta$.

\begin{prop}\label{ctyest}
Let $F$ be a bounded subset of $\rn$ and let $0< \theta<\phi  \leq 1$. Then 
\begin{equation}\label{ineqs2A}
\uda F \ \leq \ \udb F \ \leq \  \frac\phi\theta\ \uda F
\end{equation}
and 
\begin{equation}\label{ineqs2}
\uda F \ \leq \ \udb F \ \leq \  \uda F +\Big(1-\frac{\theta}{\phi}\Big) (n- \uda F),
\end{equation}
with corresponding inequalities where $\uda$ and $\udb$ are  replaced by  $\da$ and $\db$.
\end{prop}

\begin{proof}
We include the proof of \eqref{ineqs2A} to give a feel for this type of argument. The left-hand inequality  is just monotonicity of $\uda F$.

 With $0< \theta<\phi  \leq 1$
 let ${\displaystyle t> \frac\phi\theta\, \uda F }$ and choose $s$ such that ${\displaystyle \uda F <s <\frac{\theta}{\phi}t}$.
  Given $\epsilon >0$, for all sufficiently  small $0<\delta<1$ we may find countable or finite covers $\{U_i\}_{i\in I}$ of $F$ such that
\begin{equation}\label{epsum}
\sum_{i\in I} |U_i|^s<\epsilon \quad \mbox{ and } \quad \delta \leq |U_i| \leq \delta^\theta  \quad \mbox{ for all } i\in I.
\end{equation} 
 Let
$$I_0 =\{i\in I: \delta \leq |U_i| < \delta^{\theta/\phi}\} \quad \mbox{ and } \quad I_1 =\{i\in I: \delta^{\theta/\phi} \leq |U_i| \leq \delta^\theta\}.$$ 
For each $i\in I_0$ let $V_i$ be a set with $V_i\supset U_i$ and $|V_i| = \delta^{\theta/\phi}$. 
Let $0<s<t\theta/\phi\leq n$. Then $\{W_i\}_{i\in I}:= \{V_i\}_{i\in I_0} \cup \{U_{i}\}_{i\in I_1}$ is a cover of $F$ 
by sets with diameters in the range $[\delta^{\theta/\phi} , \delta^\theta]$. 
Taking sums with respect to this cover:
\begin{align}
\sum_{i\in I}|W_i|^t \ =\  &\sum_{i\in I_0}|V_i|^t  +  \sum_{i\in I_1}|U_{i}|^t 
\ =\ \sum_{i\in I_0}\delta^{t\,\theta/\phi}  +  \sum_{i\in I_1}  |U_i|^t \nonumber\\
&\leq\  \sum_{i\in I_0}|U_i|^{t\,\theta/\phi}   +  \sum_{i\in I_1}  |U_i|^{t\,\theta/\phi} 
\ =\  \sum_{i\in I}|U_i|^{t\,\theta/\phi}\ \leq\ \sum_{i\in I}|U_i|^{s}<\epsilon. \label{sumbound}
\end{align}
Thus for all ${\displaystyle t> \frac\phi\theta\, \uda F }$, for all $\epsilon>0$, for all sufficiently small $\delta$ (equivalently, for all sufficiently small $\delta^\theta$) there is a  cover $\{W_i\}_i$ of $F$ by sets with  $(\delta^\theta)^{1/\phi} \leq |W_i| \leq \delta^\theta$ satisfying \eqref{sumbound}, so 
$\udb F  \leq {\displaystyle \frac\phi\theta\, \uda F }$. 

The analogue of \eqref{ineqs2A} for $\da$ follows by exactly the same argument by choosing covers of $F$ with $\delta \leq |U_i| \leq \delta^\theta $ for arbitrarily small $\delta$.

The proof of \eqref{ineqs2} is given in \cite{FFK}: essentially, given a cover of $F$ by sets $\{U_i\}$ with  $\delta \leq |U_i| \leq \delta^\theta $ one breaks up those $U_i$ with $\delta^\phi \leq |U_i| \leq \delta^\theta $ into smaller pieces to get a cover of $F$ by sets with diameters in the range $[\delta, \delta^\phi ]$.
\end{proof}
 
Note that the right hand inequality of \eqref{ineqs2A} is stronger than that in  \eqref{ineqs2}   precisely when  ${\displaystyle \frac{\theta}{\phi} \leq \frac{n}{\udb  F} -1}$, which is the case for all $0<\theta <\phi \leq 1$ if $\udb F \leq \frac{1}{2} n$; similarly for lower dimensions.
 
Inequality \eqref{ineqs2A} implies that ${\displaystyle \frac{\uda F}{\theta}}$ and ${\displaystyle \frac{\da F}{\theta}}$ are monotonic decreasing in $\theta \in (0,1]$;  Banaji \cite[Proposition 3.9]{Ban} points out that they are strictly decreasing if $\ubd F> 0$, respectively $\lbd F> 0$. Thus the graphs of $\theta \mapsto \uda F$ and  $\theta \mapsto \da F \, (0<\theta \leq 1)$ are starshaped with respect to the origin (i.e. each half-line from the origin in the first quadrant cuts the graphs in a single point). 

The following corollary is immediate.

\begin{cor}
The maps  $\theta \mapsto\da F$ and $\theta \mapsto\uda F$ are continuous for $\theta \in (0,1]$.
\end{cor}

By setting $\phi =1$ in Proposition  \ref{ctyest} and rearranging we get useful comparisons with box-counting dimensions.

\begin{cor}\label{genlbc}
Let $F$ be a bounded subset of $\rn$. Then 
\begin{equation}\label{ineqsbcA}
\uda F \ \geq \ n-  \frac{\big(n-\ubd F\big)}{\theta}
\end{equation}
and 
\begin{equation}\label{ineqsbc}
\uda F \ \geq \ \theta\, \ubd F,
\end{equation}
with corresponding inequalities where $\uda$ and $\ubd$ are  replaced by  $\da$ and $\lbd$.
\end{cor}
Again \eqref{ineqsbc} gives a better lower bound than \eqref{ineqsbcA}   if and only if  ${\displaystyle \theta \leq \frac{n}{\ubd F} -1}$ which is the case for all $\theta \in (0,1]$ if $\ubd F \leq \frac{1}{2} n$, and similarly for lower dimensions.

Intermediate dimensions may or may not be continuous when $\theta=0$, see Section \ref{examplessimple} for examples. Indeed, determining whether a given set has intermediate dimensions that are continuous at $\theta=0$, which relates to the distribution of scales of covering sets for Hausdorff and box dimensions, is one of the key questions in this subject. 

Banaji \cite{Ban} introduced a generalisation of intermediate dimensions  by replacing the condition $\delta^{1/\theta} \leq  |U_i| \leq \delta$ in Definition \ref{adef} by $\Phi(\delta) \leq  |U_i| \leq \delta$, where   $\Phi: (0,Y)\to \mathbb{R}$ is monotonic and satisfies $\lim_{\delta \searrow 0} \Phi(\delta)/\delta = 0$ for some $Y>0$, to obtain families of dimensions $\underline\dim^{\Phi}F$ and $\overline\dim^{\Phi}F$; clearly when $\Phi(x) = x^{1/\theta}$ we recover $\da F$ and $\uda F$. He provides an extensive analysis of these $\Phi$-{\em intermediate dimensions}. In particular they interpolate all the way between Hausdorff and box-dimensions, that is there exist such functions $\Phi^s$ for  $s\in [\hdd F, \lbd F]$ that are increasing with $s$ with respect to a natural ordering and are such that $\overline\dim^{\Phi}F=s$ and 
$\overline\dim^{\Phi}F=\min\{s, \lbd F\}$, see \cite[Theorem 6.1]{Ban}.

\section{Some tools for intermediate dimension}\label{tools}

As with other notions of dimension, there are some basic techniques that are useful for studying  intermediate dimensions and calculating them in specific cases. 

\subsection{A mass distribution principle}

The \emph{mass distribution principle}  is frequently used for finding lower bounds for Hausdorff dimension by considering local behaviour of measures supported on the set, see \cite[Principle 4.2]{Fa}.  Here are the natural analogues for $\da$ and $\uda$ which are proved using an easy modification of the standard proof for Hausdorff dimensions.

\begin{prop}{\rm \cite[Proposition 2.2]{FFK}}\label{mdp}
Let $F$ be a Borel subset of $\rn$ and let  $0\leq \theta \leq 1$ and $s\geq 0$. Suppose that there are numbers $a, c >0$ such that for arbitrarily small $\delta>0$  we can find a Borel measure $\mu_\delta$ supported on $F$ such that
$\mu_\delta (F) \geq a $, and with
\begin{equation}\label{mdiscond}
\mu_\delta (U) \leq c|U|^s \quad \mbox{ for all Borel sets  $U \subset \rn $ with } \delta \leq |U|\leq \delta^\theta.
\end{equation}
Then  $\uda F \geq s$.  Alternatively, if measures $\mu_\delta$ with the above properties can be found for all sufficiently  small $\delta$, then   $\da F \geq s$. 
\end{prop}

Note that in  Proposition \ref{mdp} a different measure $\mu_\delta$ is used for each $\delta$, but  it is essential that they all assign mass at least $a>0$ to $F$. In practice $\mu_\delta$ is often a finite sum of point masses.  

\subsection{A Frostman type lemma}

\emph{Frostman's lemma} is another powerful tool in fractal geometry which is a sort of dual to Proposition \ref{mdp}. We state here a version for intermediate dimensions. As usual $B(x,r)$ denotes the closed ball of centre $x$ and radius $r$.

\begin{prop}{\rm \cite[Proposition 2.3]{FFK}}\label{frostman}
Let $F$ be a compact subset of $\rn$, let $0< \theta \leq 1$, and let $0< s< \da F$.  Then there exists $c >0$ such that for all $\delta \in (0,1)$  there is a Borel probability measure  $\mu_\delta$ supported on $F$ such that for all $x \in \rn$ and $\delta^{1/\theta} \leq r \leq \delta$, 
\be
\mu_\delta (B(x,r)) \leq c r^s .\label{frostineq}
\ee
\end{prop}

Fraser has pointed out a nice alternative proof of  \eqref{ineqs2A} using the Frostman's lemma and the mass distribution principle. Briefly, let $0<\theta <\phi \leq 1$. 
if $s< \underline{\dim}_\phi F$, Proposition \ref{frostman} gives probability measures $\mu_\delta$  on $F$ (which we may take to be compact) such that $\mu_\delta(B(x,r)) \leq c  r^s$ for $\delta^{1/\phi}\leq r\leq \delta$. 
If  $\delta^{1/\theta} \leq r\leq \delta^{1/\phi}$ then
$$ \mu_\delta(B(x,r))\leq \mu_\delta(B(x,\delta^{1/\phi}))\leq c\,\delta^{s/\phi} \leq c\, r^{s\theta/\phi},$$
so $ \mu_\delta(B(x,r))\leq c\, r^{s\theta/\phi}$ for all $\delta^{1/\theta} \leq r\leq \delta$. Using Proposition \ref{mdp} $\da F\geq  s\phi/\theta$. This is true for all $s< \underline{\dim}_\phi F$ so $\da F\geq \frac{\theta}{\phi}\underline{\dim}_\phi F$.

\subsection{Relationship with Assouad dimension}

Assouad dimension has been studied intensively in recent years, see the books \cite{FraBook,robinson} and paper \cite{Fra}. Although Assouad dimension does not {\it a priori} seem closely related to intermediate dimensions, it turns out that information about the Assouad dimension of a set can refine estimates of intermediate dimensions and under certain conditions imply discontinuity at  $\theta=0$.  

The \emph{Assouad dimension} of $F \subset \rn$ is defined by
\begin{eqnarray*}
\dim_\textup{A} F &=& \inf \Big\{ s \geq 0  \ : \ \text{there exists $C>0$ such that $N_r(F \cap B(x,R)) \leq C \Big(\frac{R}{r} \Big)^s$ } \\
&\,& \qquad \qquad   \text{for all $x \in F$ and all $0<r<R$}  \Big\},
\end{eqnarray*}
where $N_r(A)$ denotes the smallest number of sets of diameter at most $r$ that can cover a set $A$.  In general  $\underline{\dim}_\textup{B} F \leq \overline{\dim}_\textup{B} F \leq \dim_\textup{A} F \leq n$, but equality of these three dimensions often occurs, even if the Hausdorff dimension and box-counting dimension differ, for example if the box-counting dimension is equal to  the ambient spatial dimension. 

The following proposition due to Banaji, which extends an earlier estimate in  \cite[Proposition 2.4]{FFK}, gives lower bounds for intermediate dimensions in terms of Assouad and box dimensions. This lower bound is sharp, taking $F$ to be the $F_p$ of Section \ref{secseq}, and can be particular useful near $\theta=1$ where the estimate approaches the box dimension.

\begin{prop}{\rm \cite[Proposition 3.10]{Ban}} \label{assouad}
For a bounded set $F \subset \rn$ and $\theta \in (0,1]$, 
\[
\da F \geq  \frac{\theta\dim_\textup{A}\! F\, \underline{\dim}_\textup{B} F}{\dim_\textup{A} F -(1-\theta) \underline{\dim}_\textup{B} F},
\]
with a similar inequality for upper dimensions. In particular, if $\underline{\dim}_\textup{B} F = \dim_\textup{A} F$ $($which is always the case if $\underline{\dim}_\textup{B} F = n$$)$, then $\da F =\overline{\dim}_\theta F = \underline{\dim}_\textup{B} F = \dim_\textup{A} F$ for all $\theta \in (0,1]$. 
\end{prop}

One consequence of  Proposition \ref{assouad} is that if $\dim_\textup{H} F < \underline{\dim}_\textup{B} F = \dim_\textup{A} F$, then the intermediate dimensions $\da F$ and $\uda F$ are  constant on $(0,1]$ and discontinuous at $\theta = 0$.  This will help us analyse examples that exhibit a range of behaviours in Section \ref{examplessimple}.  

Banaji also shows \cite[Proposition 3.8]{Ban}  that \eqref{ineqs2A}, \eqref{ineqs2} and  \eqref{ineqsbcA} may be strengthened by incorporating the Assouad dimension of $F$ into the right-hand estimates. 

\subsection{Product formulae}

It is natural to relate dimensions of products of sets to those of the sets themselves. The following product formulae for  intermediate dimensions are of interest in their own right and are also useful in constructing examples.

\begin{prop}{\rm \cite[Proposition 2.5]{FFK}} \label{products}
Let $E \subset \rn$ and $F \subset \mathbb{R}^m$ be bounded and let $\theta \in [0,1]$.  Then
\be\label{prodineq}
\da E + \da F\ \leq\ \da (E \times F) \    \leq\ \overline{\dim}_\theta ( E \times F)\ \leq\ \overline{\dim}_\theta  E + \overline{\dim}_\textup{B} F.
\ee
\end{prop}

\noindent {\it Sketch proof.}
The cases $\theta=0,1$ are well-known, see \cite[Chapter 7]{Fa}.
For other $\theta$ the left hand inequality follows by using Proposition \ref{frostman} to put measures on $E$ and $F$ satisfying inequalities of the form \eqref{frostineq} and then applying Proposition \ref{mdp} to the product of these two measures.

The middle inequality is trivial. For the right hand inequality let $s > \overline{\dim}_\theta  E$ and $d > \overline{\dim}_\textup{B} F$.  We can find a cover of $E$ by sets $\{U_i\}$ with $\delta^{1/\theta} \leq |U_i| \leq \delta$ for all $i$ and with $\sum_{i} |U_i|^s \leq \epsilon$. Then, for  each $i$, we find a cover $\{U_{i,j}\}_{j}$ of $F$ by  at most $|U_i|^{-d}$ sets with diameters $|U_{i,j}|=|U_{i}|$ for all $j$.  Thus
$ E\times F\ \subset\  \bigcup_i \bigcup_j \big(U_i \times U_{i,j}\big)$ 
where $\delta^{1/\theta} \leq |U_i \times U_{i,j}| \leq \sqrt{2}\delta$ for all $i,j$.
A simple estimate gives
$\sum_i\sum_j |U_i \times U_{i,j}|^{s+d} \leq\  2^{(s+d)/2}\epsilon$, leading to the right hand inequality.
\hfill $\Box$
\medskip

Banaji \cite[Theorem 5.5]{Ban} extends such product inequalities to $\Phi$-intermediate dimensions.

\section{Some examples}\label{egs}
\setcounter{equation}{0}
\setcounter{theo}{0}

The following basic examples in $\mathbb{R}$ or $\mathbb{R}^2$ serve to give a feel for intermediate dimensions and indicate some possible behaviours of $\da $ and $ \overline{\dim}_\theta$ as $\theta$ varies. 

\subsection{Convergent sequences}\label{secseq}
The {\it $p${\em th} power sequence} for  $p>0$ is given by
\begin{equation}\label{fp}
F_p = \Big\{0, \frac{1}{1^p}, \frac{1}{2^p},\frac{1}{3^p},\ldots \Big\}.
\end{equation}
Since $F_p$ is countable $\hdd F_p=0$ and a standard exercise shows that $\bdd F_p=1/(p+1)$, see \cite[Chapter 2]{Fa}. We  obtain the intermediate dimensions of $F_p$.
\begin{prop}{\rm \cite[Proposition 3.1]{FFK}}\label{conseq}
For $p>0$ and $0\leq \theta \leq 1$,
\be\label{pex}
\da F_p  = \overline{\dim}_\theta F_p = \frac{\theta}{p+\theta}.
\ee
\end{prop}

\noindent {\it Sketch proof.}
This is clearly valid when $\theta =0$. 
Otherwise, to bound  $\overline{\dim}_\theta F_p$ from above, let $0<\delta<1$ and let $M =\lceil \delta^{-(s +\theta(1-s))/(p+1)}\rceil$.
Take a covering ${\mathcal U}$ of  $F_p$ consisting of the $M$ intervals $B(k^{-p}, \delta/2)$ of length $\delta$ for $1\leq k\leq M$ together with  $\lceil M^{-p} /\delta^\theta\rceil \leq M^{-p}/\delta^\theta +1$ intervals  of length $\delta^\theta$  that cover the left hand interval $[0, M^{-p}].$   Then
\begin{eqnarray} 
\sum_{U\in {\mathcal U}} |U|^s &\leq &  M\delta ^s + \delta^{\theta s}\Big(\frac{1}{M^p\delta^\theta}+ 1\Big)\label{optm}\\ 
& \leq & 2\delta^{(\theta (s-1)+sp)/(p+1)} +  \delta^s+ \delta^{\theta s} \ \to \ 0 \nonumber
\end{eqnarray} 
as $\delta \to 0$ if $s(\theta +p) > \theta$. Thus $\overline{\dim}_\theta F_p \leq\theta/(p+\theta)$.  
[Note that  $M$ was chosen essentially to minimise the expression  \eqref{optm} for given $\delta$.]
\medskip

For the lower bound we put a suitable measure on $F_p$ and apply Proposition \ref{mdp}. Let $s = \theta/(p+\theta)$ and $0<\delta < 1$ and, as with the upper bound, let $M =\lceil \delta^{-(s +\theta(1-s))/(p+1)}\rceil$.  Define $\mu_\delta$ as the sum of point masses on the points $1/k^p \ (1\leq k<\infty)$ with
\begin{equation}\label{mass}
 \mu_\delta \Big(\Big\{\frac{1}{k^p}\Big\}\Big)\  = \ 
\left\{
\begin{array}{cl}
 \delta^s & \mbox{ if } 1\leq k \leq  M   \\
 0 &    \mbox{ if } M+1\leq k <\infty  
\end{array}
\right. .
\end{equation}
Then
$$ 
\mu_\delta(F_p) \ =\  M\delta^s 
\ \geq\ \delta^{-(s +\theta(1-s))/(p+1)}\delta^s = 1 
$$
by the choice of $s$. 
To check \eqref{mdiscond}, note that 
the gap between any two points of $F_p$ carrying mass is at least $p/M^{p+1}$.
A set  $U$ such that $\delta \leq |U|\leq \delta^\theta$,
 intersects at most 
$1+ |U|/(p/M^{p+1})  = 1+|U|M^{p+1}/p $ of the points of $F_p$ which have mass $\delta^s$.
Hence
\begin{eqnarray*} 
\mu_\delta (U)\ \leq \ \delta^s + \frac{1}{p}|U|\delta^s \delta^{-(s +\theta(1-s))} 
\ \leq\  \Big(1+ \frac{1}{p}\Big)|U|^s, 
\end{eqnarray*}
 Proposition  \ref{mdp} gives $\da F_p \geq s = \theta/(p+\theta)$.
\hfill $\Box$
\bigskip

Here is a generalisation of Proposition \ref{conseq} to sequences with `decreasing gaps'.
Let $a\in \mathbb{R}$ and let $f: [a,\infty) \to (0,1]$ be continuously differentiable with  $f'(x)$ negative and increasing and $f(x) \to 0$ as $ x\to\infty$. Considering integer values, the mean value theorem gives that $f(n)-f(n+1)$ is decreasing, so the sequence $\{f(n)\}_n$ is a `decreasing sequence with decreasing gaps'.

\begin{prop}\label{conseqgaps}
With $f$ as above, let
$$F = \big\{0, f(1), f(2), \ldots \big\}.$$
Suppose that 
${\displaystyle \frac{xf'(x)}{f(x)} \to -p}$ as $ x\to\infty$, where $0\leq p\leq \infty$.
Then for all $0< \theta \leq 1$,
$$\da F = \overline{\dim}_\theta F= \frac{\theta}{p+\theta},$$
taking this expression to be $0$ when $p=\infty$.
\end{prop}

This may be proved in a similar way to Proposition \ref{conseq} using that $xf'(x)/f(x)$ is close to, rather than equal to, $-p$ when $x$ is large.

For example, taking $f(x) = 1/\log(x+1)$, the sequence
\be\label{logs}
F_{\log} = \Big\{ 0, \frac1{\log 2}, \frac1{\log 3}, \frac1{\log 4}, \ldots \Big\}
\ee
has $\dth F_{\log} =  1$ if $\theta \in (0,1]$ and ${\mbox{\rm dim}}_{\,0}F_{\log} =0$, so there is a discontinuity at 0. On the other hand,      with $f(x) = e^{-x}$,
$$F_{\exp} = \big\{ 0, e^{-1}, e^{-2}, e^{-3}, \ldots\big\}$$
has $\dth F_{\exp} =  0$ for all $\theta \in [0,1]$.

\subsection{Simple examples illustrating different behaviours} \label{examplessimple}

Using the examples above together with tools from Section \ref{tools} we can build up simple examples of sets exhibiting various behaviours as $\theta$ ranges over $[0,1]$, shown in Figure 1.
\medskip

\begin{figure}[t]
\centering
\includegraphics[width=0.95\textwidth]{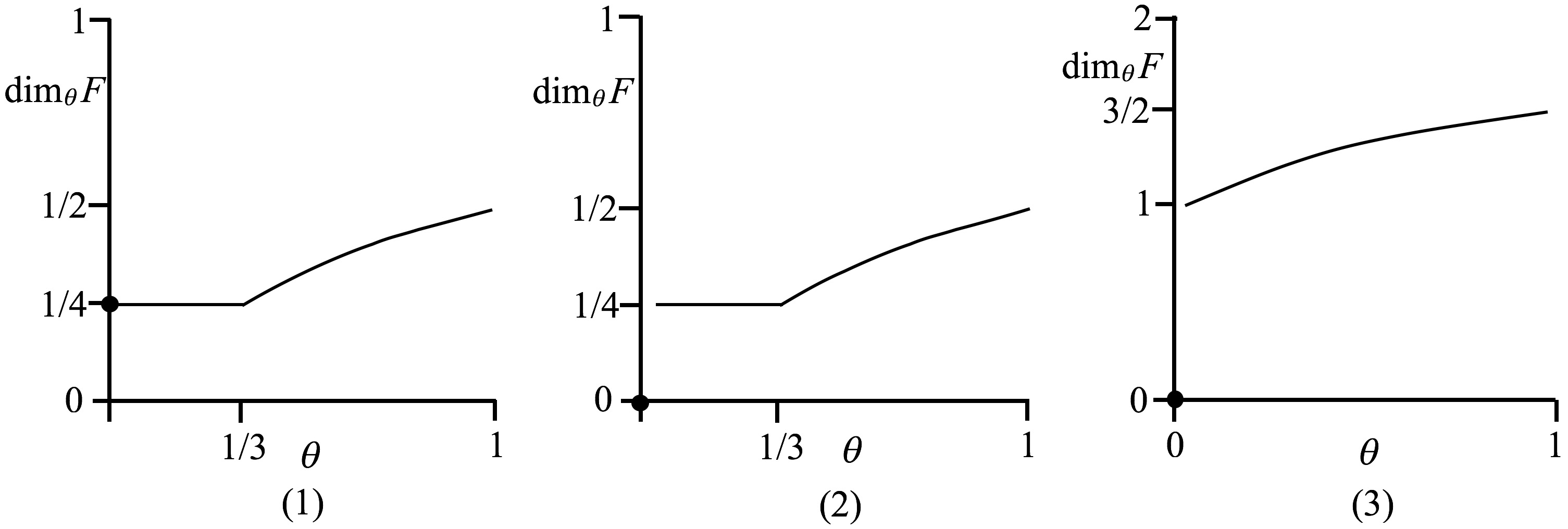}
\caption{Graphs of  $\da F$ for the three examples in Section \ref{examplessimple}}
\end{figure}

\noindent\emph{Example 1: Continuous at $0$, part constant, then strictly increasing. }  Let $F =F_{1} \cup E$ where $F_{1}$ is as in \eqref{fp} and let $E \subset \mathbb{R}$ be any compact set with $\hdd E = \ubd E = 1/4$ (for example a suitable self-similar set).  Then  
\[
\dth F =  \max\Big\{ \frac{ \theta}{1+ \theta},  \, 1/4\Big\}\qquad (\theta \in [0,1]).
\]
This follows using \eqref{pex} and the finite stability of upper intermediate dimensions. 
\medskip

\noindent \emph{Example 2:  Discontinuous at $0$, part constant, then strictly increasing.}  Let $F = F_{1}\cup E$ where this time $E\subset \mathbb{R}$ is any closed countable set with $\lbd E = \dim_\textup{A} E = 1/4$.  Using Proposition \ref{assouad} and finite stability of upper intermediate dimensions, 
\medskip

\[
\dth F =  \max\Big\{ \frac{ \theta}{1+ \theta},  \, 1/4\Big\} \qquad (\theta \in (0,1].
\]
Note that the intermediate dimensions are exactly as in Example 1 except when $\theta=0$ and a discontinuity occurs.
\medskip

\noindent \emph{Example 3: Discontinuous at $0$, smooth and strictly increasing.} Consider the countable set
\[
F = F_{1} \times F_{\log}\subset \mathbb{R}^2.
\]
Then  ${\mbox{\rm dim}}_{0}F= \hdd F=0$ and
\[
\dth F =  \frac{ \theta}{1 + \theta}+ 1 \qquad (\theta \in (0,1]),
\]
noting that  $\dth F_{\log}  = \bdd F_{\log}  = \dim_\textup{A} F_{\log}= 1$ for $\theta \in (0,1]$
using \eqref{logs} and Propositions  \ref{assouad} and \ref{products}.

\subsection{Circles, spheres and spirals} \label{spirals}

Infinite sequences of concentric circles and spheres with radii tending to 0 might be thought of as higher dimensional analogues of the sets $F_p$ defined in \eqref{fp}. A countable union of concentric circles will have Hausdorff dimension 1, but the box and intermediate dimensions may be greater as a result of the accumulation of circles at the centre. For $p>0$ define the family of circles 
$$C_p = \big\{x \in \mathbb{R}^2: |x| \in F_p\big\}.$$
Tan \cite{Tan} showed, using the mass distribution principle and the Frostman lemma, Proposition \ref{frostman}, that
$$\da C_p = \overline{\dim}_\theta C_p= 
\left\{
\begin{array}{ll} 
 \frac{2p+2\theta(1-p)}{2p+\theta(1-p)}  & \mbox{\rm if } 0<p\leq 1     \\
  1 &  \mbox{\rm if } 1\leq p     \\   
\end{array}
\right.
$$
with analogous formulae for concentric spheres in $\mathbb{R}^n$ and also for families of circles or spheres with radii given by other monotonic sequences converging to 0. He also considers  families of points evenly distributed across such sequences of circles or spheres for which the intermediate dimension may be discontinuous at 0.

Closely related to circles are spirals. For $0< p\leq q$ define  
$$S_{p,q} = \bigg\{\bigg(\frac{1}{t^p} \sin \pi t, \frac{1}{t^q} \cos \pi t \bigg):t\geq 1\bigg\} \subset\mathbb{R}^2.$$
Then $S_{p,q}$ is a spiral winding into the origin, if $p=q$ it is a circular polynomial spiral, otherwise it is an elliptical polynomial spiral.  Burrell, Falconer and Fraser \cite{BFF2} calculated that  
$$\da S_{p,q} = \overline{\dim}_\theta S_{p,q}= 
\left\{
\begin{array}{ll}
 \frac{p+q+2\theta(1-p)}{p+q+\theta(1-p)}  & \mbox{\rm if } 0<p\leq 1\\
  1 &  \mbox{\rm if } 1\leq p     \\   
\end{array}
\right. .
$$
Not unexpectedly, when $p=q$ these circular polynomial spirals have the same intermediate dimensions as the concentric circles $C_p$.

Another variant is the `topologist's sine curve' given, for $p>0$ by 
$$T_p = \bigg\{\bigg(\frac{1}{t^p}, \sin \pi t \bigg):t\geq 1\bigg\} \subset \mathbb{R}^2,$$
that is the graph of the function $f:(0,1] \to \mathbb{R}$ given by $f(x) = \sin(\pi x^{-1/p})$.
Tan \cite{Tan} used related methods show that 
$$\da T_p = \overline{\dim}_\theta T_p = \frac{p+2\theta}{p+\theta},$$
as well as finding the intermediate dimensions of various generalisations of this curve.

\section{Bedford-McMullen carpets}\label{bmc}

Self affine carpets are a well-studied class of fractals where the Hausdorff and box-counting dimensions generally differ; this is a consequence of the alignment of the component rectangles in the iterated construction. The dimensions of planar self-affine carpets were first investigated by Bedford \cite{bedford} and McMullen \cite{mcmullen} independently, see also \cite{peres}, and these carpets have been widely studied and generalised, see \cite{Fa1,FraSur} and references therein.  Finding the intermediate dimensions of these carpets gives information about the range of scales of covering sets needed to realise their Hausdorff and box-counting dimensions. Deriving exact formulae seems a major challenge, but some lower and upper bounds have been obtained, in particular enough to demonstrate continuity of the intermediate dimensions at $\theta=0$ and that they attain a strict  minimum when $\theta=0$. 

\begin{figure}[t]
\centering
\includegraphics[width=0.8\textwidth]{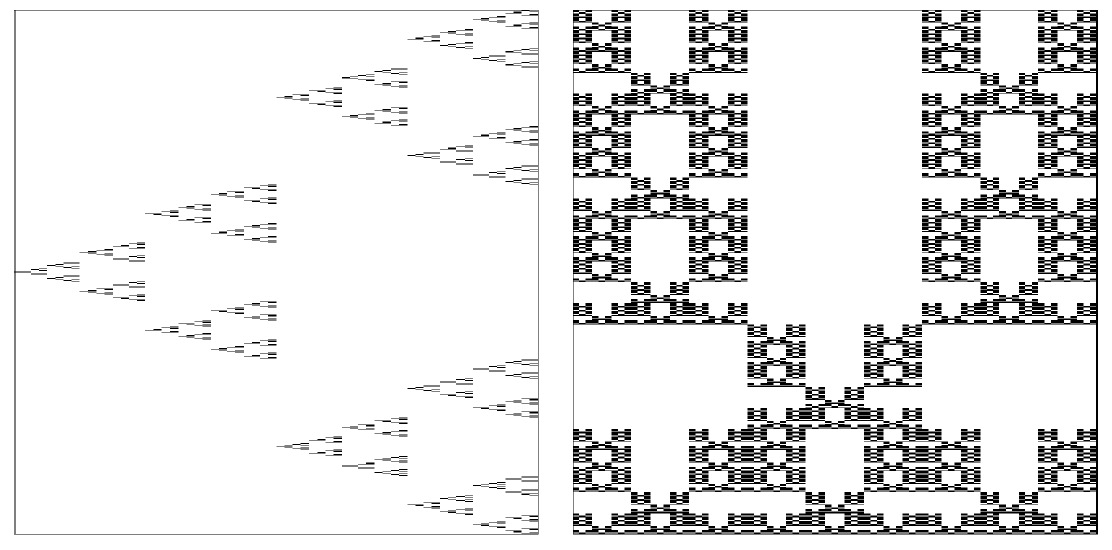}
\caption{A $2\times3$ and a $3\times5$ Bedford-McMullen carpet}
\end{figure}

Bedford-McMullen carpets are attractors of iterated function systems of a set of affine contractions, all translates of each other which preserve horizontal and vertical directions. More precisely, for integers $n > m \geq 2$, an $m\times n$-carpet is defined in the following way.
Let ${I}=\left\{0,\ldots, m-1 \right\}$ and ${J}=\left\{0,\ldots, n-1 \right\}$ and let $D \subset I \times J$ be a {\it digit set} with at least two elements.  For each  $\left( p,q \right)\in D$ we define the affine contraction $S_{\left(p,q \right)}\colon [0,1]^2 \rightarrow [0,1]^2$ by
\[
S_{ \left( p,q \right)}\left(x,y\right) = \left( \frac{x+p}{m},\frac{y+q}{n} \right) .
\]
Then $ \left\{ S_{ \left(p,q \right)} \right\}_{ \left( p,q \right)\in D}$ is an iterated function system so there exists a unique non-empty compact set $F \subset [0,1]^2$ satisfying 
\[
F=\bigcup_{(p,q) \in D}S_{ \left( p,q \right)}(F)
\]
called a {\em Bedford-McMullen self-affine carpet}, see Figure 2 for examples.  The carpet can also be thought of as the set constructed using a `template' consisting of the selected rectangles $ \left\{ S_{ \left(p,q \right)}([0,1]^2) \right\}_{ \left( p,q \right)\in D}$ by repeatedly substituting affine copies of the template in each of the selected rectangles. 

Bedford \cite{bedford} and McMullen \cite{mcmullen} showed that the box-counting dimension of $F$ exists with
\be\label{dimb}
\bdd F = \frac{\log M}{\log m} + \frac{\log N - \log M}{\log n} 
\ee
where $N$ is the total number of selected rectangles and $M$ is the number of $p$ such that there is a $q$ with $(p,q) \in D$, that is the number of columns of the template containing at least one rectangle.
They also showed that 
\be\label{dimh}
\hdd F = \frac{\log \big(\sum_{p=1}^m N_p^{\log_n m} \big)}{\log m},
\ee
where $N_p \, (1\leq p\leq m)$ is the number of $q$ such that $(p,q)\in D$, that is the number of  rectangles  in the $p$th column of the template. The Hausdorff and box-counting dimensions of $F$ are equal if and only if the number of selected rectangles in every non-empty column is constant. 

Virtually all work on these carpets depends on dividing the iterated rectangles into `approximate squares'. The box-counting dimension result \eqref{dimb} is then a straightforward counting argument. The Hausdorff dimension \eqref{dimb} argument is more involved; McMullen's approach defined a Bernoulli-type measure $\mu$ on $F$ via the iterated rectangles and obtained an upper bound for the local upper density of $\mu$ that is valid everywhere and a lower bound valid $\mu$-almost everywhere. These ideas have been adapted and extended for estimating intermediate dimensions, but with the considerable complication that one seeks good density estimates that are valid over a restricted range of scales, but even getting close estimates for the intermediate dimensions seems a considerable challenge.

The best upper bounds known at the time of writing are:
\be\label{upbound}
\overline{\dim}_\theta F \ \leq \  \hdd F +  \bigg(\frac{2\log(\log_m n)\log a}{\log n}\bigg) \frac{1}{-\log \theta}
\quad \big(0<\theta <{\textstyle\frac{1}{4}}(\log_n\! m)^2\big),
 \ee
proved in \cite{FFK}. The $-1/\log\theta$ term makes this  a very poor upper bound as $\theta$ increases away from 0, but at least  it implies that $ \underline{\dim}_\theta F$ and $ \overline{\dim}_\theta F$ are continuous at $\theta = 0$ and so are continuous on $[0,1]$. 
 An upper bound for $\theta$ that is better except close to 0  was given in \cite{Kol}:
\be\label{upbound2}
 \overline{\dim}_\theta F \ \leq \  \bdd F   - \frac{ \Delta_0(\theta)}{\log n} (1-\theta) <\bdd F  \quad (\log_n m\leq\theta <1),
 \ee
where $\Delta_0(\theta)$ is the solution an equation involving a large deviation rate term which can be found numerically in particular cases. This upper bound is strictly increasing near 1 and by monotonicity also gives a constant upper bound if  $0< \theta <\log_n m$.

A reasonable lower bound that is linear in $\theta$ is
\begin{equation}\label{lowerbnd}
\underline{\dim}_{\theta}F\geq \hdd F+\theta\frac{\log |D|-H(\mu)}{\log n} \quad (0\leq \theta \leq 1),
\end{equation}
where $H(\mu)$ is the entropy of McMullen's measure $\mu$; this was essentially proved in \cite{FFK}, but see \cite{Kol} for a note on the constant. In particular this
implies that there is a strict minimum for the intermediate dimensions at $\theta =0$.  An alternative lower bound depending on optimising a certain function was given by \cite{Kol}: 
\begin{equation}\label{lowerbnd2}
\underline{\dim}_{\theta}F\geq \sup_{t>0} \psi(t,\theta) \quad (0\leq \theta \leq 1)
\end{equation}
Here $\psi(t,\theta)$ depends on entropies of linear interpolants of probability measures of the form 
$\theta^t \widetilde{{\bf p}} + (1-\theta^t )\widehat{{\bf p}}$ and $\theta^t \widetilde{{\bf q}} + (1-\theta^t )\widehat{{\bf q}}$ where $\widetilde{{\bf p}}, \widetilde{{\bf q}}$ and $\widehat{{\bf p}},\widehat{{\bf q}}$ are measures that occur naturally in the calculations for, respectively, the box-counting and Hausdorff dimensions of the carpets. Of course, the lower bounds given by Corollary \ref{genlbc} for a general $F$ in terms of box-counting dimensions also apply here. In particular, Banaji's general lower bound \cite[Proposition 3.10]{Ban} in terms of the box and Assouad dimensions of $F$ gives the best-known lower bound for $\theta$ close to 1 for some, though not all, Bedford-McMullen carpets.

Many questions on the intermediate dimensions of these carpets remain, most notably finding the exact forms of 
$\underline{\dim}_{\theta}F$ and $\overline{\dim}_{\theta}F$. Towards that we would at least conjecture that the lower and upper intermediate dimensions are equal and strictly monotonic.

\section{Potential-theoretic formulation}

The potential-theoretic approach for estimating Hausdorff dimensions goes back to Kaufman \cite{Kau}. More recently box-counting dimensions have been defined in terms of energies and potentials with respect to suitable kernels and these have been used to obtain results on the box-counting dimensions of projections of sets in terms of `dimension profiles', see \cite{fal:2020a,fal:2020b}. In particular the box-counting dimension of the projection of a Borel set $F\subset \rn$ onto $m$-dimensional subspaces is constant for almost all subspaces (with respect to the natural invariant measure) generalising the long-standing results of  Marstrand \cite{mar:1954} and Mattila \cite{mat:1975} for Hausdorff dimensions. 

As with Hausdorff and box-counting dimensions, it turns out that $\theta$-intermediate dimensions can be characterised in terms of capacities with respect to certain kernels, and this can be extremely useful as will be seen in Section \ref{secproj}.
Let $\theta \in (0, 1]$ and $0<m \leq n$ ($m$ is often an integer, though it need not be so). For $0 \leq s \leq m$ and $0<r <1$, define the  kernels
\be\label{ker}
{\phi}_{r, \theta}^{s, m}(x) = \begin{cases} 
      1 & 0\leq |x| < r \\
      \big(\frac{r}{|x|}\big)^s & r\leq |x| < r^\theta   \\
      \frac{r^{\theta(m-s) + s}}{|x|^m}\ & r^\theta \leq |x|
   \end{cases} \qquad (x\in \rn).
\ee
If  $s = m$ this reduces to 
\be\label{kerb}
\phi_{r, \theta}^{m, m}(x) = \begin{cases} 
      1 &0\leq |x| < r \\
      \big(\frac{r}{|x|}\big)^m & r \leq |x| 
   \end{cases}\qquad (x\in \rn),
\ee 
which are the kernels $\phi_r^m(x)$ used   in the context of box-counting dimensions \cite{fal:2020a,fal:2020b}.   Note that  ${\phi}_{r, \theta}^{s, m}(x)$ is continuous in $x$ and monotonically decreasing in $|x|$.     Let $\mathcal{M}(F)$ denote the set of Borel probability measures supported on a compact $F\subset \rn$. The \emph{energy} of $\mu \in \mathcal{M}(F)$ with respect to $\phi_{r, \theta}^{s, m}$ is  
\begin{equation}\label{endef}
\int\int \phi_{r, \theta}^{s, m}(x - y) \,d\mu(x)d\mu(y)
\end{equation}
and the {\em potential} of $\mu$ at $x \in \rn$ is
\begin{equation}\label{potdef}
\int \phi_{r, \theta}^{s, m}(x - y)\,d\mu(y).
\end{equation}
The \emph{capacity} $C_{r, \theta}^{s, m}(F)$ of $F$ is the reciprocal of the minimum energy achieved by probability measures on $F$, that is
\begin{equation}\label{capdef}
{C_{r, \theta}^{s, m}(F)} = \left(\inf\limits_{\mu \in \mathcal{M}(E)} \int\int \phi_{r, \theta}^{s, m}(x - y) \,d\mu(x)d\mu(y)\right)^{-1}.
\end{equation}
Since $\phi_{r, \theta}^{s, m}(x)$ is continuous in $x$ and strictly positive and $F$ is compact, $C_{r,\theta}^{s, m}(F)$ is positive and finite. For general bounded sets we take the capacity  of a set to be that of its closure.\\

The existence of energy minimising measures and the relationship between the minimal energy and the corresponding potentials is standard in classical potential theory, see \cite[Lemma 2.1]{fal:2020a} and \cite{BFF} in this setting.   
In particular, there exists an {\em equilibrium measure} $\mu \in \mathcal{M}(E)$ for which the energy \eqref{endef}
attains a minimum value, say $\gamma$. Moreover, the potential \eqref{potdef} of this equilibrium measure is at least $ \gamma$
for all $x \in F$ (otherwise perturbing $\mu$ by a point mass where the potential is less than $\gamma$ reduces the energy) with equality for $\mu$-almost all $x \in F$. These properties turn out to be key in expressing these dimensions in  terms of capacities.

Let $F \subset \rn$ be compact,  $m\in (0,n]$, $\theta \in (0, 1]$ and  $r\in (0,1)$. It may be shown that 
\be\label{decrease}
\frac{\log C_{r, \theta}^{s, m}(F)}{-\log r}\, -\, s
\ee
is continuous in $s$ and decreases monotonically from positive when $s=0$ to negative or 0 when $s=m$. Thus there is a unique $s$ for which \eqref{decrease} equals 0. Moreover, the rate of decrease of  \eqref{decrease} is bounded away from 0 and from $-\infty$ uniformly for $r\in (0,1)$. 
This means we can pass to the limit as $r\to 0$ and for each  $m\in(0,n]$ define the \emph{lower $\theta$-intermediate dimension profile} of $F \subset \rn$ as
\be\label{lidp}
\lid^m F =  \textnormal{ the unique } s\in [0,m] \textnormal{ such that  } \lowlim\limits_{r \rightarrow 0}\frac{\log C_{r, \theta}^{s, m}(F)}{-\log r} = s
\ee
and the \emph{upper $\theta$-intermediate dimension profile} as
\be\label{ludp}
\uid^m F =  \textnormal{ the unique } s\in [0,m] \textnormal{ such that  } \uplim\limits_{r \rightarrow 0}\frac{\log C_{r, \theta}^{s, m}(F)}{-\log r} = s.
\ee

Since the kernels $ \phi_{r, \theta}^{t, m}(x)$ are decreasing in $m$ the intermediate dimension profiles \eqref{lidp} and \eqref{ludp}  are increasing in $m$.

The reason for introducing \eqref{lidp} and \eqref{ludp} is that they not only permit an equivalent definition of $\theta$-intermediate dimensions but also give the intermediate dimensions of the images of sets under certain mappings, as we will see in Section \ref{secproj}. The following theorem states the equivalence  between intermediate dimensions when defined by sums of powers of diameters as in Definition \ref{adef} and using this
capacity formulation.

\begin{theo}\label{equivdim}
Let $F \subset \rn$ be bounded and $\theta \in (0, 1]$. Then
\begin{equation*}
\lid F = \lid^{n} F
\end{equation*}
and
\begin{equation*}
\uid F = \uid^{n} F.
\end{equation*}
\end{theo}

The proof of these identities involve relating the potentials to $s$-power sums of diameters of  covering balls of $F$ with diameters in the required range, using a decomposition into annuli to relate this to the kernels, see \cite[Section 4] {BFF}.

We defined the intermediate dimension profiles  $\lid^m F$ and $\uid^m F$  for $F\subset \rn$ but Theorem \ref{equivdim} refers just to the case when $m=n$. The significance of these dimension profiles  when $0<m<n$  will become clear in the next section.

\section{Projections and other images}\label{secproj}

The relationship between the dimensions of a set $F \subset \rn$ and its orthogonal projections $\pi_V(F)$ onto  subspaces $V\in G(n,m)$, where $G(n,m)$ is the Grassmannian of $m$-dimensional subspaces of $\rn$ and $\pi_V: \rn \to V$ denotes orthogonal projection, goes back to the foundational work on Hausdorff dimension  by Marstrand \cite{mar:1954} for $G(2,1)$ and Mattila \cite{mat:1975} for general $G(n,m)$. They showed that for a Borel set $F \subset \rn$ 
\be\label{marmat}
\hdd \pi_V(F)  = \min\{\hdd F, m\}  
\ee
for almost all $m$-dimensional subspaces $V$ with respect to the natural invariant probability measure $\gamma_{n,m}$ on $G(n,m)$, where $\hdd$ denotes Hausdorff dimension. Later Kaufman \cite{Kau} gave a potential-theoretic proof of these results.
See, for example, \cite{FFJ} for a survey of the many generalisations, specialisations and consequences of these projection results. In particular, there are theorems that guarantee that the lower and upper box-counting dimensions and the packing dimensions of the projections $\pi_V(F)$ are constant for almost all $V\in  G(n,m)$, see \cite{fal:2020a,fal:2020b,FH2,How}. This constant value is {\it not} the direct analogue of \eqref{marmat} but rather it is given by a dimension profile of $F$. 

Thus a natural question is whether there is a Marstrand-Mattila-type theorem for intermediate dimensions, and it turns out that this is the case with the $\theta$-intermediate dimension profiles $\lid^m F$ and $\uid^m F$ defined in \eqref{lidp} and \eqref{ludp} providing the almost sure values for orthogonal projections from $\rn$ onto $m$-dimensional subspaces. 
Intuitively, we think of  $\lid^m F$ and $\uid^m F$ as the intermediate dimensions of $F$ when regarded from an $m$-dimensional viewpoint. 

\begin{theo}\label{main}
Let $F \subset \rn$ be bounded. Then, for  all $V \in G(n, m)$
\begin{equation}\label{mains}
\lid \pi_V F \leq \lid^{m} F \quad\mathrm{and}\quad \uid \pi_V F \leq \uid^{m} F
\end{equation}
for all    $\theta\in (0,1]$.  Moreover, for $\gamma_{n, m}$-almost all $V \in G(n, m)$,
\begin{equation}\label{mainas}
\lid \pi_V F = \lid^{m} F \quad\mathrm{and}\quad \uid \pi_V F = \uid^{m} F
\end{equation}
for  all $\theta\in (0,1]$.
\end{theo}

The upper bounds  in \eqref{mains}  utilise the fact that orthogonal projection does not increase distances, so does not increase the values taken by the kernels, that is 
$$\phi_{r, \theta}^{s, m}(\pi_V x-\pi_V y)\geq \phi_{r, \theta}^{s, m}(x-y)  \qquad (x,y\in  \rn).$$
By comparing the energy of the equilibrium measure on $F$ with its projections onto each $\pi_V F$ it follows that $C_{r, \theta}^{s, m}(\pi_V F)\geq C_{r, \theta}^{s, m}(F) $ and using \eqref{lidp} or \eqref{ludp} gives the $\theta$-intermediate dimensions of $\pi_V F$ as a subset of the $m$-dimensional space $V$.

The almost sure lower bounds in \eqref{mainas} essentially depend on the relationship between the kernels and on $\rn$ and on their averages over  $V \in G(n,m)$. More specifically, 
for  $m \in \{1, \dots, n-1\}$ and $0 \leq s < m$ there is a constant $a> 0$, depending only on $n,m$ and $s$, such that for all $x \in \rn$, $\theta \in (0, 1)$ and $0< r < \frac{1}{2}$,
$$
\int \phi_{r, \theta}^{s, m}(\pi_V x- \pi_V y) d\gamma_{n, m}(V) \leq \ a\,\phi_{r, \theta}^{s, m}(x-y) \log \frac{r}{|x-y|}.
$$
Using this for a sequence $r= 2^{-k}$  with a Borel-Cantelli argument gives  \eqref{mainas}. Full details may be found in \cite[Section 5]{BFF}.
\medskip

Theorem \ref{main} has various consequences, firstly concerning continuity at  $\theta=0$.
\begin{cor} \label{cor1}
Let $F \subset \mathbb{R}^n$ be such that $\lid F$ is continuous at $\theta=0$. Then $\lid \pi_V F$ is continuous at $\theta =0$ for almost all $V$. A similar result holds for the upper intermediate dimensions.
\end{cor}
\begin{proof}
If $\hdd F \geq m$ then for almost all $V$, $\hdd \pi_V(F) = m = \lid \pi_V F$ for all $\theta \in [0,1]$ by \eqref{marmat}. Otherwise, for almost all $V$ and all $\theta \in [0,1]$,
$$\hdd F = \hdd\pi_V F\leq \lid \pi_V F \leq \lid^m F\leq \lid F \to \hdd F$$
as $\theta \to 0$, where we have used \eqref{marmat} and \eqref{mains}.
\end{proof}

For example, taking  $F\subset \mathbb{R}^2$ to be an $m\times n$ Bedford-McMullen carpet (see Section \ref{bmc}), it follows from \eqref{upbound} and Corollary \eqref{cor1} that the intermediate dimensions of projections of $F$ onto almost all lines are continuous at 0.  In fact more is true: if $\log m /\log n \notin \mathbb{Q}$ then $\lid \pi_V F$ and $\uid \pi_V F$ are continuous at 0 for projections onto {\em all} lines $V$, see \cite[Corollaries 6.1 and 6.2]{BFF} for more details.

The following surprising corollary shows that continuity of intermediate dimensions of a set at 0 is enough to imply a relationship between the {\em Hausdorff dimension} of a set and  the {\em box-counting} dimensions of its projections.

\begin{cor}\label{cor4}
Let $F \subset \mathbb{R}^n$ be a bounded set such that $\lid F$ is continuous at $\theta=0$.  Then
\[
\lbd \pi_V F = m
\]
for almost all $V \in G(n, m)$ if and only if
\[
\hdd F \geq m.
\]
A similar result holds on replacing lower by upper dimensions.
\end{cor}
\begin{proof}
The `if' direction is clear even without the continuity assumption, since if $\hdd F \geq m$, then 
\[
m \geq \lbd \pi_V F  \geq \hdd \pi_V F  \geq m
\]
for all $V$ using \eqref{marmat}.  

On the other hand, suppose  that $\lbd \pi_V F = m$ for almost all $V$.  The final statement of Proposition \ref{assouad} gives that $\lid  \pi_V F = m$ for all $\theta \in (0,1]$ for almost all $V$.  As $\lid F$  is assumed continuous at $\theta=0$, Corollary \ref{cor1} implies that $\lid  \pi_V F$ is continuous at  $0$ for almost all $V$ and so 
$\hdd F =\hdd \pi_V F = \underline{\mbox{\rm dim}}_{0} \pi_V F=m$ for almost all $V$, using \eqref{marmat}.
\end{proof}

An striking  example of this is given by products of the sequence sets $F_p$ of \eqref{fp} for $p >0$.  By Proposition \ref{conseq} $\bdd F_p = \theta/(\theta+p)$  so by Proposition \ref{products}
\[
\mbox{\rm dim}_{\theta} (F_p \times F_p) = \frac{2\theta}{\theta+p} \quad (\theta \in [0,1]),
\]
which is continuous at $\theta =0$. Since $\hdd (F_p \times F_p) =0$,  Corollary \ref{cor4} implies that 
\[
\ubd \pi_V (F_p \times F_p) < 1
\]
for  almost all $V$. This is particularly striking when $p$ is close to 0, since $\bdd (F_p \times F_p) =2/(1+p)$ is close to 2 but still the box-counting dimensions of its projections never reach 1. In fact, a calculation not unlike that in Proposition \ref{conseq} shows that 
 for all projections onto lines $V$, apart from the horizontal and vertical projections,
\[
\ubd \pi_V( F_p \times F_p)  =  1- \left(\frac{p}{p+1}\right)^2.
\]

Analogous ideas using dimension profiles can be be used to find dimensions of images of a given set $F$ under other parameterised families of mappings. These include images under certain stochastic processes (which are parameterised by points in the probability space). For example, let $B_\alpha: \mathbb{R} \to \mathbb{R}^m $ be index-$\alpha$ fractional Brownian motion where $0<\alpha<1$, see for example \cite[Section 16.3]{Fa}.  The following theorem generalises the result of Kahane \cite{Kau} on the Hausdorff dimension of fractional Brownian images and that of Xiao \cite{Xi} for  box-counting and packing dimensions of fractional Brownian images. 

\begin{theo}\label{brown}
Let $F \subset \rn$ be compact. Then, almost surely, for all $0\leq \theta\leq 1$,
\begin{equation}\label{brownid}
\lid B_\alpha(F) = \frac1\alpha \lid^{m\alpha} F \quad\mathrm{and}\quad \lid B_\alpha(F) = \frac1\alpha \lid^{m\alpha} F.
\end{equation}
\end{theo}
 The proof of this is along the same lines as for projections, see \cite{Bur} for details. The upper bound  uses that for all $\epsilon >0$ fractional Brownian motion satisfies an almost sure H\"{o}lder condition $|B_\alpha(x) - B_\alpha(y)|\leq M|x-y|^{1/2 -\epsilon} $ for $x,y\in F$, where $M$ is a random constant. The almost sure lower bound uses that
$$\mathbb{E} \big(\phi_{r, \theta}^{s m} (B_\alpha(x) - B_\alpha(y))\big)\ \leq\ c\, \phi_{r, \theta}^{s m}(x-y)$$
where $c$ depends only on $m$ and $s$.

We can get an explicit form of the intermediate dimensions of these Brownian images taking $F=F_p$ of \eqref{fp}.

\begin{prop}\label{brownfp}
For index-$\alpha$ Brownian motion $B_\alpha: \mathbb{R} \to \mathbb{R}$, almost surely, for all $0\leq \theta\leq 1$ and $p>0$,
\begin{equation}\label{brownidfp}
\lid B_\alpha(F_p)\ =\ \uid B_\alpha(F_p)\ =\ \frac{\theta}{p\alpha+ \theta}.
\end{equation}
\end{prop}

In particular \eqref{brownidfp} is less than the upper bound $\theta /\alpha(p+\theta)$ that comes from directly applying the almost sure  H\"{o}lder condition \eqref{Holder} for $B_\alpha$ to the intermediate dimensions of $F_p$.  

\section{Open problems}

Finally here are a few open questions relating to intermediate dimensions. A general problem is to find the possible forms of intermediate dimension functions. At the very least they are constrained by the inequalities of Proposition \ref{ctyest}.

\begin{question}  Characterise the possible functions $\theta\mapsto \da F$ and $\theta\mapsto \uda F$ that may be realised by some set $F \subset \mathbb{R}$ or $F \subset \mathbb{R}^n$.
\end{question} 

It may be easier to answer more specific questions about the form of the dimension functions. I am not aware of any counter-example to the following suggestion. 

\begin{question}  Is it true that if $\uda F$, respectively $\da F$, is constant for  $\theta \in [a,b] $ where $ 0<a<b\leq 1$ then it must be constant for $\theta \in (0,b] $?
\end{question} 

Similarly, the following question suggested by Banaji seems open.
\begin{question}  Can $\uda F$ or $\da F$ be convex functions of $\theta$, or even (non-constant) linear functions? 
\end{question} 

As far as I know, in all cases where explicit values have been found, the intermediate dimensions equal upper bounds obtained using coverings by sets of just the two diameters $\delta^{1/\theta}$ and $\delta$ (or constant multiples thereof). It seems unlikely that this is enough for every set, indeed Kolossv\'{a}ry \cite[Section 5]{Kol} suggests that three or more diameters of covering sets may be needed to get close upper bounds for the intermediate dimensions of Bedford-McMullen carpets.

\begin{question}  Are there (preferably fairly simple) examples of sets $F$ for which the intermediate dimensions $\uda F$or $\da F$ cannot be approximated from above using coverings by sets just of two diameters? Are there even sets where the number of different scales of covering sets needed to get arbitrary close approximations to the intermediate dimensions is unbounded?
\end{question}

Coming to more particular examples, the Bedford-McMullen carpets are a class of sets where current knowledge of the intermediate dimensions is limited.

\begin{question}  Find the exact form of the intermediate dimensions $\da F$ and $\uda F$ for the Bedford McMullen carpets $F$ discussed in Section 5, or at least improve the existing bounds.
\end{question} 

Getting exact formulae for these dimensions is likely to be challenging, but better bounds, in particular the asymptotic form near $\theta =0$ and $\theta =1$, would be of interest. It would also be useful to know more about the behaviour of the intermediate dimensions of these carpets as functions of $\theta$.

\begin{question}  Are the intermediate dimensions $\da F$ and $\uda F$ of Bedford McMullen carpets $F$ equal? Are they strictly increasing in $\theta$? Are they differentiable, or even analytic, as functions of $\theta$ or can they exhibit phase transitions?
\end{question}



\section*{Acknowledgements}
The author thanks Amlan Banaji, Stuart Burrell, Jonathan Fraser, Tom Kempton and Istv\'{a}n Kolossv\'{a}ry  for many discussions around this topic. The work was supported in part by an EPSRC Standard Grant EP/R015104/1.
\\

\bibliographystyle{plain}

\end{document}